\documentclass{amsart}
\usepackage{graphicx}
\usepackage[active]{srcltx}
\usepackage{amsthm,mathrsfs}
\usepackage{a4wide}
\usepackage[english]{babel}
\usepackage{amsfonts}
\usepackage{amsmath}
\usepackage{amssymb}
\usepackage{dsfont}
\newtheorem{lemma}{Lemma}
\newtheorem{theorem}{Theorem}
\newtheorem{corollary}{Corollary}

\usepackage{float}
\usepackage{url}

\newcommand {\p} {\mathbb{P}}

\newcommand {\N} {\mathbb{N}}

\newcommand {\fu} {\mathfrak{u}}
\newcommand {\afu} {{|\mathfrak{u}|}}

\newcommand {\ve} {\varepsilon}

\makeatletter\def\blfootnote{\xdef\@thefnmark{}\@footnotetext}\makeatother
\allowdisplaybreaks
\title{\bf Tractability results for the weighted star-discrepancy}

\author{Christoph Aistleitner} 
\address{Department of Applied Mathematics, School of Mathematics and Statistics, University of New South Wales, Sydney NSW 2052, Australia}
\email{aistleitner@math.tugraz.at}

\thanks{The author is supported by a Schr\"odinger scholarship of the Austrian Research
Foundation (FWF)}


\begin{document}

\begin{abstract}
The weighted star-discrepancy has been introduced by Sloan and Wo{\'z}niakowski to reflect the fact that in multidimensional integration problems some coordinates of a function may be more important than others. It provides upper bounds for the error of multidimensional numerical integration algorithms for functions belonging to weighted function spaces of Sobolev type. In the present paper, we prove several tractability results for the weighted star-discrepancy. In particular, we obtain rather sharp sufficient conditions under which the weighted star-discrepancy is strongly tractable. The proofs are probabilistic, and use empirical process theory.
\end{abstract}

\date{}
\maketitle

\section{Introduction}

For a set of points $x_1, \dots, x_N$ from the $d$-dimensional unit cube $[0,1]^d$, for any $z \in [0,1]^d$ the \emph{discrepancy function} $\Delta(z)$ is defined as
$$
\Delta(z) = \frac{1}{N} \sum_{n=1}^N \mathds{1}_{[0,z)} (x_n) - \lambda([0,z)),
$$
and the star-discrepancy $D_N^*(x_1, \dots, x_N)$ is defined as 
$$
D_N^*(x_1, \dots, x_N) = \sup_{z \in [0,1]^d} \left| \Delta(z) \right|.
$$
Here $[0,z)$ is an axis-parallel box that stretches from the origin to $z$, and $\lambda$ denotes the ($d$-dimensional) Lebesgue measure. The Koksma--Hlawka inequality states that for a function $f$ on $[0,1]^d$ the difference between the arithmetic mean of the function values $f(x_1), \dots, f(x_N)$ and the integral of $f$ over $[0,1]^d$ is bounded by the star-discrepancy of $x_1, \dots, x_N$, multiplied with the (Hardy--Krause) variation of $f$ over $[0,1]^d$. Consequently, point sets having small star-discrepancy can be used to approximate a multidimensional integral. This method for numerical integration is an example of the so-called \emph{quasi-Monte Carlo (QMC) method} , which uses cleverly constructed deterministic point sets as sampling points (as opposed to the Monte Carlo method, where randomly sampled points are used).\\

There exist several constructions of point sets achieving a discrepancy of order $\ll (\log N)^{d-1} N^{-1}$, for fixed $d$ and for $N \to \infty$. However, these bounds are only useful if the number of points $N$ is very large (i.e., at least exponential) in comparison with $d$, which means that QMC integration using such points is not feasible on a computer if $d$ is large. To describe the problem concerning the existence of low-discrepancy point sets of moderate cardinality for large values of $d$, the notion of the \emph{inverse of the star-discrepancy} can be used. Let $n^*(d,\ve)$ denote the smallest possible cardinality of a point set in $[0,1]^d$ having discrepancy at most $\ve$. By a result of Heinrich \emph{et al.}~\cite{hnww} for any $d$ and $N$ there exist points $x_1, \dots, x_N \in [0,1]^d$ such that
\begin{equation} \label{hnww}
D_N^*(x_1, \dots, x_N) \leq c_{\textrm{abs}} \frac{\sqrt{d}}{\sqrt{N}}
\end{equation}
(where we can choose $c_{\textup{abs}}=10$, see~\cite{aist}), which implies that
$$
n^*(d,\ve) \leq c_{\textup{abs}} d \ve^{-2}
$$
($c_{\textup{abs}}$ denotes positive absolute constants, not always the same). On the other hand, Hinrichs~\cite{hinr} proved the lower bound 
$$
n^*(d,\ve) \geq c_{\textup{abs}} d \ve^{-1}.
$$
Thus there exist high-dimensional low-discrepancy point sets which have moderate cardinality in comparison with the dimension $d$. Note, however, that constructing such point sets is a largely unsolved problem (cf.~\cite{dgw,dgw2}), and that calculating (or estimating) the discrepancy of a given high-dimensional point set is generally a very difficult problem (see~\cite{gsw}).\\

A series of numerical investigations of Paskov and Traub in the mid-1990s showed that in practice QMC integration can still be successfully applied to high-dimensional problems, and often perform significantly better than what could be expected from theoretical upper bounds (see \cite{past}). One possible explanation is that often for a formally high-dimensional problem only a small number of coordinates is really important, while other (or most) coordinates are much less important. This idea led to the introduction of weighted function spaces and weighted discrepancies by Sloan and Wo{\'z}niakowski~\cite{sw}. These concepts are closely connected with the theory of (weighted) reproducing kernel Hilbert spaces of Sobolev type; in particular, the error of a QMC integration scheme for a function $f$ from such a weighted space can be estimated in terms of the norm of $f$ in this space and the corresponding weighted discrepancy of the set of sampling points, by means of a weighted Koksma--Hlawka inequality. For details, see~\cite{sw} as well as~\cite{dksh,dpd}.\\

By the expression \emph{weights} we mean a set $\gamma$ of non-negative real numbers $\gamma_\fu$, indexed by the class of all non-empty subsets $\fu$ of the set of coordinates $\{1, \dots, d\}$ (or indexed by the class of all non-empty subsets of $\N$). An important special case are \emph{product weights}, which satisfy
$$
\gamma_\fu = \prod_{j \in \fu} \gamma_j,
$$
where $\gamma_j$ is the weight of $\{j\}$, that is, the weight associated with the $j$-th coordinate.\\

Let $\afu$ denote the cardinality of $\fu$. For a point $x \in [0,1]^d$ and a non-empty subset $\fu$ of $\{1, \dots, d\}$, we write $x(\fu)$ for the $|\fu|$-dimensional point which consists only of those coordinates of $x$ whose index belongs to $\fu$. Furthermore, we write $(x(\fu);1)$ for the $d$-dimensional vector which has the same coordinates as $x$, except that coordinates whose index is not in $\fu$ are replaced by 1. Then the \emph{weighted star-discrepancy} of the points $x_1, \dots, x_N \in [0,1]^d$ for weights $\gamma = (\gamma_\fu)_{\fu \subset \{1, \dots, d\}}$ is defined as
$$
D_{N,\gamma}^*(x_1, \dots, x_N) = \sup_{z \in [0,1]^d}~ \max_{\fu \subset \{1, \dots, d\}} ~\gamma_\fu |\Delta(z(\fu);1)|.
$$
For simplicity of writing, in this definition and throughout the rest of this paper we assume that $\fu \subset A$ denotes only the \emph{non-empty} subsets of a given set $A$. \\

We also need the notions of tractability and strong tractability. Let $n_\gamma^*(d,\ve)$ denote the smallest possible cardinality of a set of points from $[0,1]^d$ whose weighted star-discrepancy (with respect to the weights $\gamma$) is at most $\ve$. Then the weighted star-discrepancy is called \emph{tractable} if there exist non-negative constants $C,\alpha,\beta$ such that
\begin{equation} \label{err}
n_\gamma^*(d,\ve) \leq C d^{\alpha} \ve^{-\beta}
\end{equation}
for all $d \geq 1$ and $\ve \in (0,1)$. Furthermore, it is called \emph{strongly tractable} if~\eqref{err} holds for $\alpha=0$. In the case of strong tractability, the infimum of $\beta$ in~\eqref{err}, for $\alpha=0$, is called the $\ve$-\emph{exponent} of strong tractability. For more details on definitions and properties of the tractability of multidimensional problems, see the monographs of Novak and Wo{\'z}niakowski~\cite{nw1,nw2,nw3}.\\

Amongst others, the following tractability results for the weighted star-discrepancy are known: 
\begin{itemize}
\item (\cite[Theorem 1]{hpst}) There exists an (unknown) absolute constant $c_{\textup{abs}}$ such that for any weights $\gamma$, any $d \geq 1$ and any $N \geq 1$ there exist points $x_1, \dots, x_N \in [0,1]^d$ whose weighted star-discrepancy is bounded by
\begin{equation} \label{hps}
D_{N,\gamma}^*(x_1, \dots, x_N) \leq c_{\textup{abs}} \frac{1 + \sqrt{\log d}}{\sqrt{N}} \max_{\fu \subset \{1, \dots, d\}} \gamma_\fu \sqrt{\afu}.
\end{equation}
\item (\cite[Theorem 3]{hpst}, based on~\cite[Corollary 8]{dnpw}) For product weights satisfying 
\begin{equation} \label{hps3}
\sum_{j=1}^\infty \gamma_j < \infty,
\end{equation}
the weighted star-discrepancy is strongly tractable with $\ve$-exponent equal to 1.
\item (\cite[Theorem 4]{hpst}) Let weights $\gamma_\fu$ be given which satisfy $\gamma_\fu \geq c$ for some constant $c$ whenever $\afu =  2$. Then for any points $x_1, \dots, x_N \in [0,1]^d$ we have
\begin{equation} \label{hps2}
D_{N,\gamma}^*(x_1, \dots, x_N) \geq \frac{c}{12},
\end{equation}
provided that the inequality $d \geq 2^{N+1}$ holds. Consequently, the logarithmic factor in~\eqref{hps} can in general not be removed.
\item (\cite[Main Theorem]{hswo}) For product weights, assume that
\begin{equation} \label{hswcond}
\sum_{j=1}^\infty \gamma_j^a <  \infty
\end{equation}
for some (arbitrary) constant $a$. Then for any $b>0$ there exists a constant $C(b)$ such that for any $d \geq 1$ and any $N \geq 1$ there are points $x_1, \dots, x_N \in [0,1]^d$ such that
\begin{equation*}
D_{N,\gamma}^*(x_1, \dots, x_N) \leq C(b) \frac{1}{N^{1/2-b}}.
\end{equation*}
\end{itemize}
It should be noted that (in)tractability results for the weighted star-discrepancy also imply similar (in)tractability results for multidimensional numerical integration in the corresponding function space; this connection is presented in detail in~\cite{hswo}.\\

The purpose of the present paper is the following. Firstly, using a recent large deviations bound for star-discrepancies due to Aistleitner and Hofer~\cite{aistho}, to establish a numerically explicit and slightly improved version of~\eqref{hps}. Secondly, to provide a sufficient condition for strong tractability for the weighted star-discrepancy for general weights $\gamma$ (not necessarily product weights). And finally, to show that in the case of product weights condition~\eqref{hswcond} is not optimal as a sufficient condition for strong tractability of the weighted star-discrepancy, and can be replaced by a much weaker condition.\\

It should be remarked that all results in this paper are non-constructive existence results, which use probabilistic arguments and empirical process theory in the same spirit as the proofs in~\cite{hnww}. However, there also exist several constructive results, for example due to Dick, Leobacher and Pillichshammer~\cite{dlpc}, Larcher, Pillichshammer and Scheicher~\cite{lpsw} and Wang~\cite{wang1,wang2}. Furthermore, the results in this paper should also be compared with tractability results for the weighted $L^p$-discrepancy, see for example Leobacher and Pillichshammer~\cite{lpb} and the references therein.

\section{Results}

The following result is a slightly improved and numerically explicit version of~\eqref{hps}. The subsequent corollary is similar to a corresponding corollary (\cite[Corollary 1]{hpst}) which follows from~\eqref{hps}. In the sequel, by $e$ we denote Euler's number.
\begin{theorem} \label{th1}
For any weights $\gamma=(\gamma_\fu)_{\fu \subset \{1, \dots, d\}}$, any $d \geq	1$ and any $N \geq 1$ there exist points $x_1,\dots, x_N \in [0,1]^d$ such that
$$
D_{N,\gamma}^*(x_1, \dots, x_N) \leq \frac{5.7}{\sqrt{N}}~ \max_{\fu \subset \{1, \dots,d\}} \gamma_{\fu} \sqrt{4.9 + 2 \log \left(\frac{ed}{|\mathfrak{u}|}\right)} \sqrt{|\mathfrak{u}|}.
$$
\end{theorem}

\begin{corollary} \label{co1}
For weights $(\gamma_\fu)_{\fu \subset \N}$, if the number
$$
C_\gamma := \sup_{d \geq 1}~ \max_{\fu \subset \{1, \dots, d\}}~ \gamma_\fu \sqrt{\afu}
$$
is finite, then for the weighted star-discrepancy of the point set from Theorem~\ref{th1} we have
$$
D_{N,\gamma}^*(x_1, \dots, x_N) \leq \frac{6 C_\gamma}{\sqrt{N}} \sqrt{5 + 2 \log (ed)},
$$
which means that for $n_\gamma^*(d,\ve)$ we have the upper bound
$$
n_\gamma^*(d,\ve) \leq \frac{36 C_\gamma^2 (5 + 2 \log(ed))}{\ve^2}.
$$
Consequently, in this case the weighted star-discrepancy is tractable.\\

Furthermore, if even
\begin{equation} \label{co1equ}
\hat{C}_\gamma := \sup_{d \geq 1}~ \max_{\fu \subset \{1, \dots, d\}}~ \gamma_\fu \sqrt{4.9 + 2 \log \left(\frac{ed}{|\mathfrak{u}|}\right)} \sqrt{\afu}
\end{equation}
is finite, then for the weighted star-discrepancy of the point set from Theorem~\ref{th1} we have
$$
D_{N,\gamma}^*(x_1, \dots, x_N) \leq \frac{6 \hat{C}_\gamma}{\sqrt{N}},
$$
and 
$$
n_{\gamma}^*(d,\ve) \leq \frac{36 \hat{C}_\gamma^2}{\ve^2}.
$$
Consequently, in this case the weighted star-discrepancy is strongly tractable with $\ve$-exponent at most $2$.\\
\end{corollary}

Note that, contrary to~\eqref{hps}, Theorem~\ref{th1} allows us to regain~\eqref{hnww} for the case of the classical star-discrepancy (with a slightly worse value for the numerical constant). In fact, the classical star-discrepancy is a special case of the weighted star-discrepancy where the only non-vanishing weight is $\gamma_{\{1, \dots,d\}}=1$, and all other weights are zero. Inserting these weights into Theorem~\ref{th1} we get a point set $x_1, \dots, x_N \in [0,1]^d$ satisfying
$$
D_N^*(x_1, \dots, x_N) \leq \frac{5.7}{\sqrt{N}} \sqrt{6.9 ~d} \leq 15 \frac{\sqrt{d}}{\sqrt{N}}.
$$
Thus any significant improvement of Theorem~\ref{th1} would include an improvement of~\eqref{hnww}, which appears to be a very difficult problem.

\begin{theorem} \label{th4}
Assume that the weights $\gamma = (\gamma_\fu)_{\fu \subset \N}$ satisfy the conditions $\gamma_\fu \leq c \afu^{-1/2},~\fu \subset \N,$ and 
\begin{eqnarray} \label{c1th4}
\sum_{\fu \subset \N} e^{-c \gamma_\fu^{-2}} < \infty
\end{eqnarray}
for some constant $c>0$. Then there exists a constant $C_\gamma$ such that for any $N \geq 1$ and $d \geq 1$ there exist points $x_1, \dots, x_N \in [0,1]^d$ for which we have
$$
D_{N,\gamma}^* (x_1, \dots, x_N) \leq \frac{C_\gamma}{\sqrt{N}}.
$$
Consequently, the weighted star-discrepancy for such weights is strongly tractable with $\ve$-exponent at most 2.
\end{theorem}
It turns out that Theorem~\ref{th4} actually contains the second part of Corollary~\ref{co1} (without the values for the numerical constants) as a special case;  a proof of this fact will be given at the end of this paper. \\

For the case of product weights, the convergence condition of Theorem~\ref{th4} can be significantly simplified. More precisely, we will obtain the following result.

\begin{theorem} \label{th2}
For product weights satisfying the condition
\begin{eqnarray} \label{c1}
\sum_{j=1}^\infty e^{-c \gamma_j^{-2}} < \infty
\end{eqnarray}
for some $c>0$ there is a constant $C_\gamma$ such that for any $N \geq 1$ and $d \geq 1$ there exist points $x_1, \dots, x_N \in [0,1]^d$ for which we have
$$
D_{N,\gamma}^* (x_1, \dots, x_N) \leq \frac{C_\gamma}{\sqrt{N}}.
$$
Consequently, the weighted star-discrepancy for such weights is strongly tractable, with $\ve$-exponent at most 2.
\end{theorem}

Note that condition~\eqref{c1} is much weaker as a sufficient condition for strong tractability than~\eqref{hswcond} (and of course also much weaker than~\eqref{hps3}, though at the expense of a worse value for the $\ve$-exponent). In fact, condition~\eqref{hswcond} resembles the conditions for strong tractability of the weighted $L^p$-discrepancy, which are typically of the form
$$
\sum_{j=1}^\infty \gamma_j^{p/2} < \infty
$$
(see~\cite{hswo2,lpb,nwi}). In~\cite{hswo}, Hickernell, Sloan, and Wasilkowski speculated that condition~\eqref{hswcond} might be necessary for strong tractability of the weighted star-discrepancy; as Theorem~\ref{th2} shows, this is not the case.\\

A typical sequence $(\gamma_j)_{j \geq 1}$ satisfying condition~\eqref{c1} of Theorem~\ref{th2} is $\gamma_j = \frac{\hat{c}}{\sqrt{\log j}}$ for some constant $\hat{c}>0$. The following Theorem \ref{th3} shows that for having strong tractability for the weighted star-discrepancy for product weights a growth condition for $\gamma_j$ of the order of a negative power of $\log j$ is necessary. More precisely, if $\gamma_j$ is not bounded above by a negative power of $\log j$, then the weighted star-discrepancy is not strongly tractable.

\begin{theorem} \label{th3}
For product weights $\gamma$, assume that the sequence $(\gamma_j)_{j \geq 1}$ is non-increasing. If for every $c>0$ and $\ve>0$ we have
\begin{equation} \label{th3c}
\gamma_j \geq \frac{c}{(\log j)^\ve}
\end{equation}
for infinitely many $j$, then the weighted star-discrepancy is \emph{not} strongly tractable.\\

Furthermore, for product weights $\gamma$ for non-increasing $(\gamma_j)_{j \geq 1}$, if there exist $c>0$ and $\ve>0$ such that we have
\begin{equation*} \label{th3c*}
\gamma_j \geq \frac{c}{(\log j)^{1/4-\ve}}
\end{equation*}
for infinitely many $j$, then the weighted star-discrepancy may be strongly tractable, but the $\ve$-exponent of tractability must be greater than 2.
\end{theorem}

There remains a gap between Theorems~\ref{th4} and~\ref{th2} and Theorem~\ref{th3}, so the exact optimal condition for strong tractability of the weighted star-discrepancy remains an open problem.

\section{Proofs}

For the proof of Theorem~\ref{th1} we will use the following result, which is~\cite[Theorem 1]{aistho}.

\begin{lemma} \label{lemma1}
For any $d \geq 1, ~N \geq 1$ and $q \in (0,1)$, a set of $N$ independent, identically distributed (i.i.d.) $[0,1]^d$-uniformly distributed random points $Y_1, \dots, Y_N$ satisfies
$$
D_N^*(Y_1, \dots, Y_N) \leq 5.7 \sqrt{4.9+\frac{\log((1-q)^{-1})}{d}} \frac{\sqrt{d}}{\sqrt{N}}
$$
with probability at least $q$.
\end{lemma}

The following lemma, which we will use for the proof of Theorem~\ref{th4} and~\ref{th2}, is a special case of~\cite[Theorem 2]{hnww}, which is the central ingredient in the original proof of~\eqref{hnww}. It follows from deep results of Talagrand~\cite{tala} and Haussler~\cite{hau}.
\begin{lemma} \label{lemma2}
There exists an absolute constant $K$ such that the following holds: Let $Y_1, \dots, Y_N$ be i.i.d. $[0,1]^d$-uniformly distributed random points. Then for all $t > 0$ we have 
$$
\p \left( D_N^*(Y_1, \dots, Y_N) \geq \frac{t}{\sqrt{N}} \right) \leq \frac{1}{t} \left( \frac{K t^2}{d} \right)^d e^{-2t^2}.
$$
\end{lemma}
(Note: in the formulation of~\cite[Theorem 2]{hnww} the additional assumption $t \geq K \sqrt{d}$ can be found; however, as~\cite[Theorem 1.1]{tala} shows, this additional assumption is not necessary.)\\

As already pointed out in~\cite{hnww}, Talagrand's result is actually much more general than Lemma~\ref{lemma2} above; as a consequence, it could be used to prove results similar to Theorem~\ref{th4} and~\ref{th2} in the present paper for more general weighted discrepancies, and not only for the weighted star-discrepancy.

\begin{proof}[Proof of Theorem~\ref{th1}]
Let $\mathcal{P} = \{z_1, \dots, z_N\}$ be a set of $N$ i.i.d. $[0,1]^d$-uniformly distributed random points, and let $\mathfrak{u}$ be a non-empty subset of $\{1, \dots, d\}$. Set $\mathcal{P}_\fu = \{z_1(\fu), \dots, z_N(\fu)\}$ and
$$
A_\mathfrak{u} = \left\{\mathcal{P}:~ D_N^* (\mathcal{P}_\mathfrak{u}) > 5.7 \sqrt{4.9 + 2 \log \left(\frac{ed}{|\mathfrak{u}|}\right)} \frac{\sqrt{|\mathfrak{u}|}}{\sqrt{N}} \right\}.
$$
Note that $z_1(\fu), \dots, z_N(\fu)$ are i.i.d. $[0,1]^\afu$-uniformly distributed points. Thus by Lemma~\ref{lemma1}, for $q = 1 - e^{-2 \afu \log(ed/\afu)}$ we have
\begin{eqnarray*}
\p \left(A_\fu \right) \leq e^{-2\afu \log(ed/\afu)}.
\end{eqnarray*}
Consequently, 
\begin{eqnarray}
\p \left( \bigcup_{\fu \subset \{1, \dots, d\}} A_\fu \right) \leq \sum_{r=1}^d ~\sum_{\substack{\fu \subset \{1, \dots, d\},\\\afu = r}} e^{-2\afu \log(ed/\afu)}. \label{fro}
\end{eqnarray}
For any $r \geq 1$, the number of subsets $\fu$ of $\{1, \dots, d\}$ which have cardinality $r$ is given by
$$
\binom{d}{r} \leq \left( \frac{e d}{r}\right)^r = e^{r \log(ed/r)}.
$$
Thus, continuing from~\eqref{fro} and noting that $\log(ed/r) \geq 1$, we have
$$
\p \left( \bigcup_{\fu \subset \{1, \dots, d\}} A_\fu \right) \leq \sum_{r=1}^d e^{r \log(ed/r)} e^{-2r \log(ed/r)} \leq \sum_{r=1}^\infty e^{-r} = (e-1)^{-1} < 1.
$$
Thus there exist points $x_1, \dots, x_N$ which avoid all sets $A_\fu$ for $\fu \subset \{1, \dots, d\}$. By the definition of the weighted star-discrepancy for these points $x_1, \dots, x_N$ we have
$$
D_{N,\gamma}^*(x_1, \dots, x_N) \leq \frac{5.7}{\sqrt{N}} ~\max_{\fu \subset \{1, \dots,d\}} ~\gamma_{\fu} \sqrt{4.9 + 2 \log \left(\frac{ed}{|\mathfrak{u}|}\right)} \sqrt{|\mathfrak{u}|},
$$
which proves Theorem~\ref{th1}.
\end{proof}

\begin{proof}[Proof of Theorem~\ref{th4}]
Let weights $\gamma=(\gamma_\fu)_{\fu \subset \N}$ satisfying the assumptions of Theorem~\ref{th4} be given, and let $\mathcal{P} = \{z_1, \dots, z_N\}$ be a set of $N$ i.i.d. $[0,1]^{d}$-uniformly distributed random points. Without loss of generality we can assume that 
\begin{equation} \label{gammae}
\gamma_\fu \leq 1, \qquad \fu \subset \N.
\end{equation}
In fact, for weights $\gamma$ satisfying~\eqref{c1th4} it is clear that~\eqref{gammae} can always be achieved by changing at most \emph{finitely} many elements of the original set $(\gamma_\fu)_{\fu \subset \N}$, and since the constant $C_\gamma$ in the conclusion of the theorem may depend on the weights $\gamma$ we may just as well assume that~\eqref{gammae} already holds for the original weights. Consequently by~\eqref{c1th4} there also exists a constant $\hat{c}$ such that
\begin{equation} \label{c2addi}
\sum_{\fu \subset \N} e^{- \hat{c} \gamma_\fu^{-2}}\leq \frac{1}{2}.
\end{equation}
We can assume that $\hat{c}$ is so large that
\begin{equation} \label{gammae2}
\hat{c} \geq 1 \qquad \textrm{and} \qquad \hat{c} \geq c^2 K,
\end{equation}
where $c$ is the constant in the statement of Theorem~\ref{th4} and where $K$ is the absolute constant in Lemma~\ref{lemma2}.\\

For a non-empty subset $\fu$ of $\{1, \dots, d\}$ for which $\gamma_\fu \neq 0$ we set $\mathcal{P}_\fu = \{z_1(\fu), \dots, z_N(\fu)\}$ and
\begin{equation} \label{afudef}
A_\fu = \left\{\mathcal{P}:~ D_N^* (\mathcal{P}_\fu) > \hat{c}^{1/2} \gamma_\fu^{-1} N^{-1/2} \right\}.
\end{equation}
If $\gamma_\fu = 0$ we simply set $A_\fu = \emptyset$. We will use Lemma~\ref{lemma2} for $t := \hat{c}^{1/2} \gamma_\fu^{-1}$. Note that~\eqref{gammae} and the first part of~\eqref{gammae2} imply that $t \geq 1$, which means that we can omit the factor $1/t$ on the right-hand side of the inequality in Lemma~\ref{lemma2}. By Lemma~\ref{lemma2} we have
\begin{eqnarray} \label{afu3}
\p (A_\fu) & \leq & \left(\frac{\hat{c} K}{\afu \gamma_\fu^{2}} \right)^\afu e^{-2 \hat{c} \gamma_\fu^{-2}}.
\end{eqnarray}
Remember that by assumption we have $\gamma_\fu \leq c \afu^{-1/2}$. Combining this with the second part of~\eqref{gammae2} we get
\begin{equation} \label{obenda}
\frac{\hat{c}}{\afu \gamma_\fu^{2}} \geq \frac{\hat{c}}{c^2} \geq K.
\end{equation}
Using the fact that $\log (K x) \leq x$ for $x \geq K$, as a consequence of~\eqref{obenda} we have
$$
\afu \log \left( \frac{\hat{c} K}{\afu \gamma_\fu^{2}} \right) \leq \hat{c} \gamma_\fu^{-2},
$$
which implies
$$
\left(\frac{\hat{c} K}{\afu \gamma_\fu^{2}} \right)^\afu \leq e^{\hat{c} \gamma_\fu^{-2}},
$$
and consequently
$$
\p(A_\fu) \leq e^{- \hat{c} \gamma_\fu^{-2}}.
$$
Together with~\eqref{c2addi} this implies that
$$
\sum_{\fu \subset \{1, \dots, d\}} \p(A_\fu) \leq \frac{1}{2}.
$$
Consequently there exists a realization $x_1, \dots, x_N$ of the random points $z_1, \dots, z_N$ which avoids all sets $A_\fu$, and for which consequently 
$$
D_{N,\gamma}^*(x_1, \dots, x_N) \leq \hat{c}^{1/2} N^{-1/2}.
$$
This proves Theorem~\ref{th4}.
\end{proof}

\begin{proof}[Proof of Theorem~\ref{th2}]
We will show that Theorem~\ref{th2} can be reduced to Theorem~\ref{th4}. We assume in the sequel that $(\gamma_j)_{j \geq 1}$ satisfies~\eqref{c1}. Similar as in the proof of Theorem~\ref{th4}, we can assume without loss of generality that the weights satisfy
\begin{equation}\label{c2alt}
\gamma_j \leq \frac{1}{2}, \quad j \geq 1, \qquad \textrm{and} \qquad \sum_{j=1}^\infty e^{-c \gamma_j^{-2} /2 } \leq \frac{1}{2}. 
\end{equation}
Then we clearly have $\gamma_\fu \leq 2^{-\afu} \leq \afu^{-1/2}$ for all $\fu \subset \N$. It remains to show that we also have
\begin{equation} \label{alsos}
\sum_{\fu \subset \N} e^{-\hat{c} \gamma_\fu^{-2}} < \infty
\end{equation}
for some appropriate constant $\hat{c}>0$. For this we will use the following inequality, which can be seen as a generalized form of the Bernoulli inequality and can be easily proved using an induction argument. Let $a_1, \dots, a_n$ be non-negative real numbers. Then 
$$
\prod_{j=1}^n (1 + a_j) \geq 1 + \sum_{j=1}^n a_j.
$$
Using this inequality together with~\eqref{c2alt} we have
\begin{eqnarray*}
\gamma_\fu^{-2} & = & \prod_{j \in \fu} \gamma_j^{-2} \\
& \geq & 1 + \sum_{j \in \fu} (\gamma_j^{-2} - 1) \\
& = & 1 - \afu + \sum_{j \in \fu} \gamma_j^{-2} \\
& \geq & \frac{\sum_{j \in \fu} \gamma_j^{-2}}{2},
\end{eqnarray*}
and consequently, for $\hat{c} = c/2$, 
$$
e^{-c \gamma_\fu^{-2}} \leq \prod_{j \in \fu} e^{- \hat{c} \gamma_j^{-2}}.
$$
Now note that for any fixed $r \geq 1$ we have
\begin{eqnarray*}
\sum_{\substack{\fu \subset \N,\\\afu = r}} ~\prod_{j \in \fu} e^{- \hat{c} \gamma_j^{-2}} & \leq & \left( \sum_{j=1}^\infty e^{- \hat{c} \gamma_j^{-2}} \right)^r \\
& \leq & \frac{1}{2^r},
\end{eqnarray*}
where the last inequality is a consequence of the second part of~\eqref{c2alt}. Thus 
\begin{eqnarray*}
\sum_{\fu \subset \N} e^{- \hat{c} \gamma_\fu^{-2}} & \leq & \sum_{r=1}^\infty ~\sum_{\substack{\fu \subset \N,\\\afu = r}} ~\prod_{j \in \fu} e^{- \hat{c} \gamma_j^{-2}} \\
& \leq & \sum_{r=1}^\infty \frac{1}{2^r} \\
& = & 1.
\end{eqnarray*}
Consequently~\eqref{alsos} is satisfied, and Theorem~\ref{th2} follows from Theorem~\ref{th4}.
\end{proof}

\begin{proof}[Proof of Theorem~\ref{th3}]
We will deduce Theorem~\ref{th3} from~\eqref{hps2}. Let product weights $\gamma$ satisfying assumption~\eqref{th3c} be given. Assume that the weighted star-discrepancy is strongly tractable, that is that there exist constants $C$ and $\beta > 0$ such that for any $d$ and $N$ there is a set $\mathcal{P}_{N,d}$ of $N$ points in $[0,1]^d$ such that
\begin{equation} \label{3a}
D_{N,\gamma}^* (\mathcal{P}_{N,d}) \leq C N^{-\beta}.
\end{equation}
Note that under assumption~\eqref{th3c} there exist infinitely many $j$ for which 
$$
\gamma_j \geq \frac{1}{(\log j)^{\beta/4}}.
$$
Since by assumption the sequence $(\gamma_j)_{j \geq 1}$ is non-increasing, this means that there are infinitely many $d$ such that
$$
\gamma_j \geq \frac{1}{(\log d)^{\beta/4}} \qquad \textrm{for} \qquad 1 \leq j \leq d,
$$
and consequently
$$
\gamma_\fu \geq \frac{1}{(\log d)^{\beta/2}}
$$
for any $\fu \subset \{1, \dots, d\}$ which satisfies $\afu = 2$. Thus by~\eqref{hps2} for such $d$ we have
\begin{equation} \label{3b}
D_{N,\gamma}^* (\mathcal{P}_{N,d}) \geq \frac{1}{12 (\log d)^{\beta/2}},
\end{equation}
provided $d \geq 2^{N+1}$. In particular, choosing $N$ such that $d = 2^{N+1}$ implies that for such $d$ we have
\begin{equation} \label{d}
\frac{1}{12 (\log d)^{\beta/2}} \leq  \frac{C}{(\log_2 (d/2))^{\beta}}.
\end{equation}
However, it is clear that~\eqref{d} cannot hold for infinitely many values of $d$. Consequently our assertion that the weighted star-discrepancy for the weights $\gamma$ is strongly tractable must be false, which proves the first part of Theorem~\ref{th3}. The second part can be shown in the same way.
\end{proof}

\begin{proof}[Proof that Theorem~\ref{th4} contains the second part of Corollary~\ref{co1} as a special case]
Assume that the constant in~\eqref{co1equ} exists. Then for any $\fu \subset \N$, writing $\max(\fu)$ for the largest element of $\fu$, we have 
$$
\gamma_\fu \leq \frac{\hat{C}_\gamma}{\sqrt{4.9+2 \log \left(\frac{e \max(\fu)}{\afu} \right)} \sqrt{\afu}}.
$$
In particular the first assumption of Theorem~\ref{th4} is satisfied. Now let $c$ be so large that $c \hat{C}_\gamma^{-2} \geq 2$. Then
\begin{eqnarray*}
\sum_{\fu \subset \N} e^{-c \gamma_\fu^{-2}} & \leq & \sum_{r=1}^\infty~ \sum_{k=r}^\infty ~\sum_{\substack{\afu = r, \\\max(\fu) = k}} e^{- 2 \left(4.9+2 \log \left(\frac{e k}{r} \right) \right) r}.
\end{eqnarray*}
Now for any fixed $k$, the number of subsets $\fu$ of $\N$ satisfying $\max(\fu)=k$ and $\afu = r$ equals $\binom{k-1}{r-1}$. Consequently by 
$$
\binom{k-1}{r-1} \leq \binom{k}{r} \leq \left( \frac{e k}{r} \right)^r
$$ 
we obtain
\begin{eqnarray}
\sum_{\fu \subset \N} e^{-c \gamma_\fu^{-2}} & \leq &  \sum_{r=1}^\infty ~ \sum_{k=r}^\infty \left( \frac{ek}{r} \right)^r e^{- 2 \left(4.9 + 2 \log \left(\frac{e k}{r} \right)\right) r} \nonumber\\
& \leq &  \sum_{r=1}^\infty e^{-9.8 r} \sum_{k=r}^\infty \left( \frac{ek}{r} \right)^{-3r}. \label{sodun}
\end{eqnarray}
To see that this sum is finite, we can use the fact that the function $f_r(x) := \left( \frac{ex}{r} \right)^{-3r}$ is monotonic decreasing in $x$, for $x \geq r$ and any fixed $r \geq	1$. Thus
$$
\sum_{k=r}^\infty \left( \frac{ek}{r} \right)^{-3r} \leq f_r(r) + \sum_{k=r+1}^\infty \int_{k-1}^k f_r(x)~dx = f_r(r) + \int_{r}^\infty f_r(x) ~dx = e^{-3r} + \frac{r e^{-3r}}{3r-1}.
$$
Consequently from~\eqref{sodun} we get
$$
\sum_{\fu \subset \N} e^{-c \gamma_\fu^{-2}} < \infty,
$$
which means that assumption~\eqref{c1th4} of Theorem~\ref{th4} is also satisfied.
\end{proof}

\section*{Acknowledgments} I want to thank Josef Dick for drawing my attention to the problems discussed in this paper, and for many helpful discussions and comments. Thanks also to Mario Ullrich for several helpful comments.

%\bibliography{Weighted_discrepancy}
%\bibliographystyle{abbrv}

\end{document}